\documentclass[aps,superscriptaddress,11pt,pre,reprint,floatfix]{revtex4-2}

\usepackage{amsmath}
\usepackage{amsthm}
\usepackage{amsfonts}
\usepackage{graphicx}
\DeclareGraphicsExtensions{.eps,.pdf}
\usepackage[colorlinks,citecolor=blue]{hyperref}
\usepackage{booktabs}
\usepackage{multirow}
\usepackage[T1]{fontenc}

\newtheorem{theorem}{Theorem}

\newtheorem{proposition}{Proposition}
\newtheorem{corollary}{Corollary}
\theoremstyle{definition}

\newtheorem{remark}{Remark}
\begin{document}

%\preprint{}

\title{Graph entropy, degree assortativity, and hierarchical structures in networks} 


\author{Fatihcan M. Atay}
\email[Author to whom correspondence should be addressed. ]{f.atay@bilkent.edu.tr}
%\homepage[]{fatihcanatay.wordpress.com}
%\thanks{}
\affiliation{Department of Mathematics, Bilkent University, 06800 Ankara, Turkey}

\author{T\"urker B{\i}y{\i}ko\u{g}lu}
\email[]{tbiyikoglu@gmail.com}
\affiliation{Mathematics Research Center, Azerbaijan State Oil and Industry University, 183 Nizami str.,
AZ1000 Baku, Azerbaijan}


\date{Preprint August 2025. Final version in Physical Review E, \href{https://doi.org/10.1103/6kxf-qn44}{https://doi.org/10.1103/6kxf-qn44}}

\begin{abstract}
We connect several notions relating the structural and dynamical
properties of a graph. Among them are the
topological entropy coming from the vertex shift, which is related
to the spectral radius of the graph's adjacency matrix, 
the Randi\'c index, and the degree assortativity.  
We show that, among all connected graphs with the same degree sequence,
the graph having maximum entropy is characterized by a hierarchical structure;
namely, it satisfies a breadth-first search ordering
with decreasing degrees (BFD-ordering for short). Consequently, the
maximum-entropy graph necessarily has high degree assortativity;
furthermore, for such a graph the degree centrality and eigenvector 
centrality coincide.
Moreover, the notion of assortativity is related to the 
general Randi\'c index. 
We prove that the graph that maximizes the Randi\'{c} index
satisfies a BFD-ordering.
For trees, the converse holds as well.
We also define a normalized Randi\'c function and show that its maximum value equals the difference of Shannon entropies of two probability distributions defined on the edges and vertices of the graph based on degree correlations. 
\end{abstract}

%\keywords{network | entropy| breadth-first search | degree correlations | Randi\'c index}

\maketitle

\section{Introduction}

Since the beginning of the recent interest in complex networks, a major theme has been to understand the relations between various network descriptors and their connection to the network's function \cite{Boccaletti-review06,Newman-review03}. 
This is a daunting task since there is a large number of quantitative measures used in network studies, each one typically capturing only a particular aspect of the network. While correlations may be observed in numerical or experimental studies, rigorous relations between them may be difficult to establish or may even be nonexistent \cite{PhysD06}.
The aim of this article is to rigorously connect several network descriptors and concepts using mathematical tools.

The central notion of the paper is that of entropy,  
which plays a fundamental role in many areas of the
physical sciences in quantifying disorder and complexity, and as such
is a particularly relevant notion for studying complex systems and
networks. 
Topological graph entropy, by its definition, is a measure of the complexity of a certain dynamical system, the so-called vertex shift, on the graph \cite{Lind-Marcus}.
Algebraically, it can be expressed in terms of the logarithm of the spectral radius of the graph (the dominant eigenvalue of the adjacency matrix)
\cite{Cvetkovic;Doob;Sachs1995,Fujii96}, which in turn is related to several other types of dynamics on graphs,
as we briefly mention in the next section. 
Therefore, it is of practical interest to characterize high- or low-entropy graphs within a suitable class.
We are in particular interested in the class $\mathcal{C}_\pi$ of graphs having a given degree sequence $\pi$,
motivated partly by the well-known procedure of re-wiring the links in a network while keeping the vertex degrees fixed, as a means of investigating the relation between network structure and function.

Another central player of the paper, which at first sight seems unrelated to the concept of entropy, is the so-called \textit{Randi\'c index} $R$,
introduced in the 1970s to study chemical compounds \cite{Randic}, and subsequently generalized to a family of indices $R_\alpha$ parametrized by a real number $\alpha$ (e.g., \cite{Bollobas;Erdos1998,Li;Shi2008}). 
We will see that the topological entropy is bounded from below by $\log R_\alpha$ (up to a multiplicative constant), so that networks with a high Randi\'c index necessarily have higher entropy. More interestingly, the Randi\'c index, 
being essentially the sum of products of vertex degrees (weighted with an exponent $\alpha$) over the edges, is naturally related to degree-degree correlations in the graph. We obtain that \textit{assortative} networks, where high-degree vertices are more likely to be connected to other high-degree vertices \cite{Newman02}, have necessarily high entropy. Stated differently, low-entropy networks are \emph{disassortative}.

Our main contribution, however, is a much more detailed characterization of the structure of maximum-Randi\'c as well as maximum-entropy graphs having a given degree sequence. 
We give a description in terms of a  
the \textit{breadth-first-search algorithm},
which is well known  
from theoretical computer science
for applications like finding vertices in a connected component or the shortest path between two nodes  \cite{Knuth}.
We use a modified version taking vertex degrees into account, and 
say that a graph has \textit{BFD-ordering} if it satisfies a breadth-first search ordering with decreasing degrees, to be defined precisely in the paper. We prove that the graphs the having largest Randi\'c index or the largest entropy in the class $\mathcal{C}_\pi$ satisfy a BFD-ordering. Moreover, 
the eigenvector centrality measure in these extreme graphs is consistent with the degree centrality.
We give a stronger result for trees, proving that the property of BFD-ordering is both necessary and sufficient for a tree to have the largest Randi\'c index  or the largest entropy among all trees with a given degree sequence. 

Finally, we define a \emph{normalized Randi\'c function} that fits better with the notion of entropy. We show that at its maximum point, the normalized Randi\'c function equals the difference of the Shannon entropies of two natural probability measures defined on the edge and vertex sets of the graph. These results provide rigorous connections between the concepts of entropy, Randi\'c index, degree correlations, and the graph structure.

\section{Topological graph entropy}

We first recall some basic notions 
related to entropy and its relation to graphs, 
and introduce the notation.
Let $G=(V,E)$ be a finite connected
graph on $n$ vertices with vertex set $V$ and edge set
$E$.
Let $A(G)=[a_{ij}]$ denote the
adjacency matrix of $G$, where $a_{ij}$ equals
1 if there is an edge from vertex $i$ to $j$ and zero otherwise.
As a simple dynamical process on the graph, consider
walking along the edges. Labeling
the vertices by distinct symbols from some symbol set $\mathcal{S}$,
let $X$ denote the set of all infinite symbol sequences $\{\dots,s_{i-1},s_{i},s_{i+1},\dots\}$,
$s_{i}\in\mathcal{S}$, that can be obtained by such walks, 
and let $\sigma:X\to X$ be the shift map
that shifts each sequence to the left by one element,
that is, $(\sigma(x))(i) = x(i+1)$  $\forall i$ and $x\in X$. 
The pair $(X,\sigma)$ defines
a dynamical system called a \emph{shift space} over the alphabet $\mathcal{S}$ (also called a \emph{vertex
shift}) \cite{Lind-Marcus}. The diversity of symbol sequences that can be produced by
walks on $G$ gives a natural measure of the complexity of the dynamics
of $(X,\sigma)$. To quantify, let $B_{k}(X)$ be the set of all blocks
of length $k$ that are present in elements of $X$. The \emph{topological
entropy} for the system is defined as \cite{Sternberg-book}
\begin{equation}
	H(X)=\lim_{k\to\infty}\frac{\log|B_{k}(X)|}{k},
	\label{symbol-ent}
\end{equation}
where $|\cdot|$ denotes the cardinality of a set\footnote{$\log$ denotes logarithm in base 2; however, the choice of the base is not important for the discussion.}. 	
We also write
this quantity as $H(G)$, and call it the \emph{topological entropy
of the graph $G$}, since it arises directly from the graph structure
of $G$.  

It is well known that the walks on the graph $G$ can be enumerated
through its adjacency matrix. Indeed, the entries $[A^{k}]_{ij}$
of the $k$th power of $A(G)$ yield the number of walks of length
$k$ from vertex $i$ to $j$. Thus denoting 
$\mathbf{1}=(1,1,\dots,1)^{\top}\in\mathbb{R}^{n}$, the quantity 
$
	N_{k}=\langle \mathbf{1},A^{k}\mathbf{1} \rangle ,
$
called the \emph{Schwarz constant} of $G$, gives the number of walks of
length $k$. Furthermore, since $A(G)$ has nonnegative
entries, Perron-Frobenius theory asserts that
$A(G)$ has a nonnegative dominant
eigenvalue, denoted $\lambda(G)$, which equals its spectral radius, and one has 
\cite{Cvetkovic;Doob;Sachs1995,Fujii96}
\[
\lim_{k\to\infty}(N_{k})^{1/k}=\lambda(G).
\]
Taking logarithms and observing that $|B_k|=N_k$ in \eqref{symbol-ent}, we reach the fundamental relation
\begin{equation} \label{H}
	H(G)=\log\lambda(G).
\end{equation}

The combinatorial concept of graph entropy  
is related
to the Shannon entropy through a variational principle, which is the
analogue of the Gibbs variational principle in statistical mechanics.
Hence, instead of counting
all possible walks on the graph without assigning any particular rules
for the walk, we consider a stochastic process by assigning a probability
distribution $\{p_{ij}\}$ to the outgoing edges of each vertex $i$.
Thus, $p_{ij}\ge0$ and $\sum_{j}p_{ij}=1$ for all $i\in V$, and
one obtains a Markov chain on $n$ states whose state transitions
are described by the stochastic matrix $P=[p_{ij}]$. 
Let $\pi$ denote its stationary distribution, $\pi P = \pi$.
The \emph{dynamical entropy} (also called the entropy rate or the entropy production) of $P$ is defined as \cite{Shannon48,Demetrius-Manke05} 
\begin{equation} \label{shannon}
	h(P)=-\sum_{i=1}^{n}\pi_{i}\sum_{j=1}^{n}p_{ij}\log p_{ij}.
\end{equation}
Note that the inner sum represents the classical Shannon entropy of the probability distribution $[p_{i1},\dots,p_{in}]$, and $h$ is the weighted average of these entropies over the states. 
Now, for a given graph $G=(V,E)$ on $n$ vertices, let $M_{G}$ denote all
compatible stochastic matrices $P$, in the sense that $p_{ij}=0$ if $ij \notin E$. 
In words, $M_{G}$ represents the set of all Markov processes
on $n$ states where the possible transitions are described by the
graph structure $G$. One can then show that \cite{Arnold94}
\begin{equation}
H(G)=\sup_{P\in M_{G}}h(P).
\label{gibbs}
\end{equation}
Furthermore, it can be shown that the supremum is achieved by the unique
stochastic matrix $P^{*}=[p_{ij}^{*}]$ defined by
\begin{equation} \label{P*}
	p_{ij}^{*}=\frac{a_{ij}f_{j}}{\lambda(G)f_{i}},
\end{equation}
where $f=(f_{1},\dots,f_{n})$ is the \emph{Perron vector} of $A(G)$, that is, the unique positive eigenvector corresponding
to the eigenvalue $\lambda(G)$ and satisfying $\sum_i f_i = 1$. The graph entropy is then the entropy
calculated by (\ref{shannon}) for the process $P^{*}$. Using \eqref{P*}
in \eqref{shannon} yields
\[ 
	H(G)=h(P^{\ast})=\log\lambda(G), 
\]
which is the same as \eqref{H}. Thus, the information-theoretic description
based on the variational principle \eqref{shannon}--\eqref{gibbs}
is equivalent to the combinatorial definition \eqref{symbol-ent}
of the graph entropy.
It is worth noting the important role played by the Perron vector $f$ in characterizing the 
%extremal 
maximum-entropy
process $P^*$.
Later we will use the Perron vector to characterize the maximum-entropy graphs within the set of all graphs having the same degree sequence.

In addition to the Markov chain formulation, there are other important examples that connect graph entropy to the dynamical processes on the graph,
especially when the adjacency matrix appears in the equations of motion (or their linearization). 
For example, in a recent discrete-time Boolean network model of genetic control \cite{Pomerance09}, the
stability of the steady-state against small perturbations
is described by the equation 
$\mathbf{y}(t+1)=qA\mathbf{y}(t)$
where $y$ is the vector of perturbations and $q\in[0,1]$ is related to the ``bias'' of the random output function
(here assumed to be the same for each vertex).
The unperturbed state
is thus asymptotically stable if the dominant eigenvalue 
of the matrix $qA$ is less than 1, 
or, by \eqref{H}, 
%\begin{equation}
$
	H(G)+\log q<0.
	\label{stab}
%\end{equation}
$
Hence, a smaller value of entropy 
indicates a more robust system, in the sense that external perturbations 
die down more rapidly or a wider range of
values for $q$ can be allowed without introducing instability.
Similar considerations apply to population dynamics, for instance 
to age-structured population models that go back to \cite{Leslie45}. 
Here, the so-called
\emph{Leslie matrix} $A$ describes the movements between age groups
through births and aging,
and its dominant eigenvalue $\lambda$ corresponds to the asymptotic growth
rate of the population. For further information on the entropic interpretation
of population dynamics, the reader is referred to \cite{Arnold94,Demetrius04}.
An extension of the evolutionary ideas of population dynamics to general networks is given in \cite{Demetrius-Manke05}, where,
using on an evolutionary model based on entropy as a selective criterion,
the authors give 
support for the hypothesis
that network entropy as defined by \eqref{H} is a quantitative measure
of robustness, and argue that the evolutionarily stable states of
evolved networks will be characterized by extremal values of graph entropy.

We note that there are other definitions of entropy for networks used in graph theory (e.g.~\cite{Simonyi-survey01,Braunstein2006}) and statistical physics (e.g.~\cite{Sole2004,Bianconi2009,Johnson2010,Anand2011,DeDomenico2016}).
The definition \eqref{H} we adopt here
is a natural choice due to its relation to dynamics as well as its close connection to several
other graph properties, 
as we explore in the following sections.
Motivated by the examples above, we will particularly be interested in the properties of maximum-entropy graphs.

\section{Entropy, spectral radius, and BFD-ordering}   \label{sec:bfd}

For the remainder of the paper we confine ourselves to simple, finite,
connected, undirected graphs.

The largest eigenvalue of the adjacency matrix, also called the 
\emph{spectral radius} 
of the graph, 
has several uses in network theory; for example,
has been related to the percolation
threshold and the appearance of a giant component \cite{Bollobas10}. 
A well-known fact from graph theory is that the spectral radius is 
non-decreasing with respect to adding edges to a graph. 
This is easily understood from the entropy interpretation of the vertex shift \eqref{symbol-ent},
but can also be proved alternatively (e.g.~\cite{Cvetkovic;Rowlinson:1990a}). 
Moreover, if $d_{\min}$, $d_{\mathrm{avg}}$, and $d_{\max}$ denote, respectively, the smallest, average, and largest
vertex degrees in a graph, then \cite{Cvetkovic;Rowlinson:1990a}
\begin{equation}	\label{deg-lambda}
	d_{\min}\le d_\mathrm{avg}\le\lambda(G)\le d_{\max}.
\end{equation}
Denoting the degree of vertex $v$ by $d_v$, a sequence $\pi=(d_{0},\ldots,d_{n-1})$
of nonnegative integers in non-increasing order is called \emph{degree
sequence} if there exists a graph $G$ with $n$ vertices for which the
$d_i$ are its vertex degrees.
The set of all connected graphs with degree sequence $\pi$ is denoted 
by $\mathcal{C}_{\pi}$. 

We introduce an ordering $\prec$ of the vertices $v_{0},\ldots,v_{n-1}$
of a graph with respect to breadth-first search as follows (see Fig.~\ref{fig:BFD}). Select
a vertex $v_{0}\in G$ with largest degree and create a sorted list
of vertices beginning with $v_{0}$; append all neighbors 
$v_{1},\ldots,v_{d_{v_{0}}}$
of $v_{0}$ sorted by decreasing degrees; then append all neighbors
of $v_{1}$ that are not already in this list; continue recursively
with $v_{2},v_{3},\ldots$ until all vertices of $G$ are processed.
In this way we build layers where each vertex $v$ in layer $i$ has
distance $i$ from root $v_{0}$. A vertex $v$ in layer $i$ is adjacent
to some vertex $w$ in layer $i-1$. We call the least one (in the
above breadth-first search) the \emph{parent} of $v$ and $v$ a child
of $w$. Finally, this ordering $\prec$ is called \emph{breadth-first
search ordering with decreasing degrees} (\emph{BFD}-ordering for
short) if the following hold for all vertices $v,w\in V$: 
\begin{enumerate} \itemsep=0pt
	\item  if $w_{1}\prec w_{2}$, then $v_{1}\prec v_{2}$ for all children $v_{1}$ of $w_{1}$ and $v_{2}$ of $w_{2}$;
	\item  if $v\prec u$, then $d_v\geq d_u$. 
\end{enumerate} 
Thus, we begin with a largest degree on top, the vertex degrees 
are non-increasing as we go down the layers, and also in each layer the vertices are sorted in non-increasing
order (Fig.~\ref{fig:BFD}).

\begin{figure}
\centering
\includegraphics[scale=0.9]{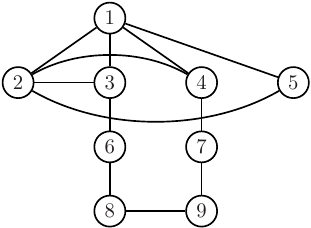} \\[6mm]
\includegraphics[scale=0.9]{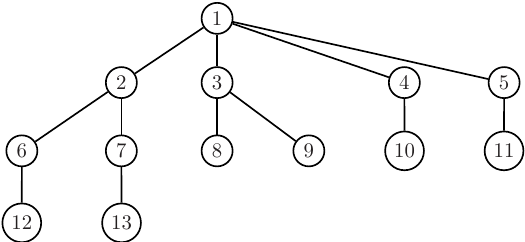}
\caption{\label{fig:BFD} Top: A BFD-ordering of the degree sequence $(4,4,3,3,2,2,2,2,2)$; 
Bottom: BFD-ordering of the tree sequence $(4,3,3,2,2,2,2,1,1,1,1,1,1)$.}
\end{figure}

We remark that for every degree sequence $\pi$, $\mathcal{C}_\pi$ has at least one graph with a BFD-ordering. This can be easily seen by the well-known result of Havel and Hakimi \cite{Havel,Hakimi}: 
A sequence of integers  $d_0, d_1\ge \cdots \ge d_n>0$, $n\ge 2$, is a degree sequence if and only if there exists a degree sequence with
$d_1-1,\ldots,d_{d_0+1}-1, d_{d_0+2},\ldots,d_n.$ Here we can choose any degree as $d_0$. By using the Havel-Hakimi condition, we can have BFD-ordering as described above, by choosing $v_0$ as the vertex with largest degree $d_0$ at the beginning and for further steps by choosing $d_1,\ldots,d_{d_0+1}$ in order, and so on. 

With this notation we can now describe the structure of the
graph having the largest entropy in a class $\mathcal{C}_{\pi}$.
Our characterization is based on the necessary condition obtained by
B\i y\i ko\u{g}lu and Leydold \cite{Biyikoglu;Leydold2008}
for a connected graph
to have the maximum spectral radius for a given degree sequence, and
the relation \eqref{H} between the spectral radius and entropy. 

\begin{theorem}[\cite{Biyikoglu;Leydold2008}]    \label{theorembfd}
	Let $G$ be the graph having the largest spectral radius 
	(and hence the largest entropy) in the class $\mathcal{C}_{\pi}$.
	Then there exists a BFD-ordering of $V(G)$ that is consistent with
	its Perron vector $f$ in such a way that $f(u)>f(v)$ 
	implies $u\prec v$ and hence $d_u\geq d_v$. 
\end{theorem}

To gain some intuitive understanding
of the graph structure implied by Theorem \ref{theorembfd},
we first note the connection to 
the concept of \emph{vertex centrality}, that is, determining which vertices in the network are the more important ones \cite{Bonacich72}. Two of the most commonly-used vertex centrality measures are the \emph{degree centrality}, which quantifies the importance of a vertex by the number of its neighbors, and \emph{eigenvector centrality}, which orders the importance of the vertices according to the entries of the Perron vector of the adjacency matrix. 
Theorem~\ref{theorembfd} states that eigenvector centrality is consistent with degree centrality in maximum-entropy networks.

Furthermore, recall formula \eqref{P*} for the extremal Markov process $P^*$ whose entropy equals the topological entropy of the graph. 
On an undirected graph, 
the process $P^*$ maximizing information entropy 
has the property that 
\begin{equation*}
\frac{p^*_{uv}}{p^*_{vu}} = \left( \frac{f_v}{f_u} \right)^{\!\! 2}
\end{equation*}
where $f$ is the Perron vector. 
For the maximum-entropy graph $G$, Theorem \ref{theorembfd} implies that 
$d_v \ge d_u$ 
if 
$ p^*_{uv} > p^*_{vu}. $
Thus, the favored directions of state transitions for the process $P^*$ are from lower- to higher-degree vertices.

The converse of Theorem \ref{theorembfd} is in general not true, since the condition of having
a BFD-ordering is in general not sufficient for maximum spectral radius
\cite{Biyikoglu;Leydold2008}. On the other hand, for trees
the condition of BFD-ordering turns out to be both
necessary and sufficient, as we will see 
in Section~\ref{sec:trees}. 
In other words, there is a unique BFD-ordering tree with maximum entropy in the class 
of all trees having a given degree sequence.

\section{The Randi\'c index and degree assortativity}  \label{sec:randic}

Theorem~\ref{theorembfd} describes the structure of maximum-entropy graphs using the concept of BFD-ordering. In this section we will show that the BFD-ordering also describes the structure of graphs that are maximal with respect to another graph measure, the so-called \emph{Randi\'c index}. This allows us to connect the concepts of graph entropy and Randi\'c index. Furthermore, we show that the latter is related to degree assortativity, thus bringing all these graph measures together.   

The \emph{Randi\'{c} index $R_{\alpha}(G)$} of a graph $G$
is defined by 
\begin{equation}  \label{randic}
R_{\alpha}(G)=\sum_{vw\in E}(d_{v}d_{w})^{\alpha}
\end{equation}
for a real parameter $\alpha$, where the sum is over the edges of the graph. Milan Randi\'{c} \cite{Randic}
originally proposed the index for $\alpha=-1$ and $\alpha=-1/2$
for measuring
the extent of branching of the carbon-atom skeleton of saturated hydrocarbons.
The case $\alpha=1$ is called the
\emph{second Zagreb index.} Bollob\'{a}s and Erd\H{o}s generalized
the index to $\alpha\in\mathbb{R}$ \cite{Bollobas;Erdos1998}. There
are several applications of the general Randi\'{c} index
related to medical and pharmacological problems \cite{Li;Shi2008}. 
In the following, we investigate its relations 
to entropy and degree correlations.

\subsection{Relation to Renyi and Tsallis entropies}

For a probability distribution $\{p_{i} \}_{i=1}^n$,
the Renyi entropy of order $\alpha$ is defined as
$$
H_{\alpha}=\frac{1}{1-\alpha}\log\left(\sum_{i=1}^{n}p_{i}^{\alpha}\right)
$$
for $\alpha\neq1$ \cite{Renyi61}. The limit $\lim_{\alpha\to1}H_{\alpha}$ exists and is equal to the Shannon entropy $-\sum_{i} p_{i}\log p_{i}$. 
Consider now a probability measure on the edge set $E$ given by
\begin{equation}
p_{E}(uv):=\frac{(d_{u}d_{v})}{\sum_{uv\in E}(d_{u}d_{v})}=\frac{(d_{u}d_{v})}{R_{1}}.\label{eq:pe}
\end{equation}
The associated Renyi entropy 
can be expressed in terms of the Randi\'c index as
\[
H_{\alpha}=\frac{1}{1-\alpha}\log\sum_{uv\in E}\frac{(d_{u}d_{v})^{\alpha}}{R_{1}^{\alpha}}=\frac{1}{1-\alpha}\log\left(\frac{R_{\alpha}}{R_{1}^{\alpha}}\right).
\]

Similarly, the Tsallis $q$-entropy of the distribution $\{p_{i}\}_{i=1}^n$ 
is defined as
$$
S_{q}=\frac{1}{q-1}\left(1 - \sum_{i=1}^n p_{i}^{q}\right)
$$
for $q\neq1$ \cite{Tsallis88}. 
The limit $\lim_{q\to 1}S_{q}$ exists and is equal to the Shannon entropy as before. 
For the probability distribution \eqref{eq:pe}, we have 
\[
S_{q}=\frac{1}{q-1}\left(1 - \frac{R_{q}}{R_{1}^{q}}\right).
\]

\subsection{Relation to graph entropy}

Our main aim is to relate the Randi\'c index to topological graph entropy $\log \lambda(G)$.
The next proposition shows that $\log R_\alpha$ essentially bounds $\log \lambda(G)=H(G)$ from below
for a given degree sequence.

\begin{proposition} \label{prop:randic-bound}
For any $\alpha \in \mathbb{R}$,
\begin{equation} \label{randic-bound}
	 \lambda(G)\ge \frac{2 R_\alpha (G)}{\sum_{v \in V} (d_v)^{2\alpha}}. 
\end{equation}
Consequently, within the class $\mathcal{C}_\pi$, graphs having a large Randi\'c index also have high topological entropy.
\end{proposition}

\begin{proof}
Consider the Rayleigh quotient on functions $f:V(G) \to \mathbb{R}$ for the largest eigenvalue: 
\begin{align*}  %\label{rayleigh}
	\lambda(G) = \max_{f\neq 0}
	\frac{\langle Af,f\rangle}{\langle f,f\rangle}
	&=\max_{f\neq 0} \frac{\sum_{v\in V}f(v)\sum_{uv\in E}f(u)}{\sum_{v\in V}f(v)^{2}} \nonumber \\
	&=\max_{f\neq 0} \frac{2\sum_{uv\in E}f(u)f(v)}{\sum_{v\in V}f(v)^{2}}.
\end{align*}
Choosing $f$ as $f(u)=(d_u)^\alpha$ establishes \eqref{randic-bound}. The second statement of the proposition follows by noting that the denominator of \eqref{randic-bound} is fixed for a given degree sequence.
\end{proof}

\begin{remark}
The expression \eqref{randic-bound} becomes an equality for certain graphs, so it cannot, in general, be improved further. We give two examples, namely, regular graphs and complete bipartite graphs. For a $d$-regular graph $G$, it is well known that the spectral radius is $\lambda(G)=d$. Moreover, 
$R_\alpha(G)=|E(G)|d^{2\alpha}$ and $\sum_{v \in V} (d_u)^{2\alpha} = d^{2\alpha} |V(G)|$, so the right-hand side of 
\eqref{randic-bound} becomes $ 2 |E(G)|/|V(G)| = d $ as well.
Hence, for regular graphs, \eqref{randic-bound} holds as  equality for any $\alpha$.
The second example is complete bipartite graphs 
$K_{s,t}$: It is well known that $\lambda(K_{s,t})=\sqrt{st}$ (see e.g. \cite{Cvetkovic;Doob;Sachs1995}). 
One partition of $K_{s,t}$ has $s$ vertices, each having vertex degree $t$, and the other partition has $t$ vertices, each having degree $s$. Also, $|E(K_{s,t})|=st$.
Hence,
$R_\alpha(K_{s,t})=st(ts)^\alpha=(st)^{\alpha+1}$
and $\sum_{v \in V} (d_u)^{2\alpha} = s(t^{2\alpha})+t(s^{2\alpha})=st(t^{2\alpha-1}+s^{2\alpha-1})$.
Therefore, the right-hand side of \eqref{randic-bound} becomes 
$\frac{2(st)^\alpha}{t^{2\alpha-1}+s^{2\alpha-1}}$, 
which is equal to $\sqrt{st}$ for $\alpha=1/2$. 
Hence, for complete bipartite graphs, \eqref{randic-bound} holds as equality for $\alpha=1/2$.
\end{remark}
 
\begin{remark}  %\label{rem:randic-bound} 
Note that $R_0$ equals the number of edges in the graph;
so the right-hand side of \eqref{randic-bound}, evaluated at $\alpha=0$, is the average degree.
Thus, the lower bound  
$\lambda(G)\ge d_{\mathrm{avg}}$ in \eqref{deg-lambda} follows from \eqref{randic-bound} with $\alpha=0$. On the other hand, taking $\alpha=\frac{1}{2}$ yields the improved estimate 
\[  \lambda(G) \ge \frac{1}{|E|}\sum_{vw\in E}\sqrt{d_{v}d_{w}}.\]
\end{remark}

\subsection{Relation to degree assortativity and BFD-ordering}

Proposition \ref{prop:randic-bound} shows that, for a given degree sequence, graphs with high Randi\'c index $R_\alpha$ also have high entropy. 
Thus, a natural way to seek high-entropy graphs 
(having a given degree sequence) 
is by increasing $R_\alpha$. 
By \eqref{randic}, for $\alpha >0$, the summands that contribute to $R_\alpha$ the most are those for which $d_u$ and $d_v$ are both large over an edge $uv\in E$. 
Hence, a high value of Randi\'c index is expected for graphs whose high-degree vertices are connected to other high-degree vertices. Interestingly, this latter property has independently attracted attention in recent literature 
as characteristic of certain real-world networks.       
Indeed, in social networks vertices with high degree are often found to be adjacent to other high-degree vertices, a tendency referred to as \emph{assortative mixing} or \emph{assortativity}, whereas technological and biological networks typically exhibit \emph{disassortative mixing}, 
where high-degree vertices are more likely to be attached to low-degree vertices \cite{Newman03}. 
Thus, the Randi\'c indices $R_\alpha$ for $\alpha>0$ act as a set of measures for the degree assortativity of networks, with  $\alpha$ quantifying the relative weight assigned to higher degree correlations. 
We note that the measure denoted $s(G)$ in \cite{Doyle-PNAS2005} actually corresponds to a Randi\'c index with $\alpha=1$. 
It has been shown in \cite{Li-Doyle05,Li-Doyle05sup} 
that $s(G)$ is closely related to the assortativity measure $r$ used in \cite{Newman02}. 
We refer the reader to \cite[Section 5.4]{Li-Doyle05} and \cite{Li-Doyle05sup} for a detailed discussion of the connection between $r$ and $s(G)$ (and hence $R_1$).

Through the Randi\'c index, we have a means of connecting degree assortativity to graph entropy using Proposition \ref{prop:randic-bound}. Recalling that 
the denominator in \eqref{randic-bound} is fixed for a given degree sequence, 
we can state the result as follows.

\begin{corollary}
Among all networks having the same degree sequence, assortative networks have high entropy, or equivalently, low-entropy networks are necessarily disassortative.	
\end{corollary}

A question that naturally arises is, how to increase the Randi\'c index within $\mathcal{C}_\pi$, i.e., alter the degree assortativity without changing the degree sequence.
There are systematic ways of moving inside the class $\mathcal{C}_\pi$, for example, by the \emph{edge switching} operation:
Recall that, if $xy,ab\in E$ are two edges of $G=(V,E)$
and $xa,yb\notin E$, switching refers to replacing the edges
$xy$ and $ab$ by $xa$ and $yb$, respectively, thus obtaining a
new graph $G'$ with the same degree sequence; see Fig.~\ref{fig:switch}.  
Now after such a switching operation, 
$R_\alpha$ changes by an amount $\Delta R_\alpha$ equal to
\[
  (d_x d_a)^\alpha + (d_y d_b)^\alpha - (d_x d_y)^\alpha - (d_a d_b)^\alpha
=(d_b^\alpha-d_x^\alpha)(d_y^\alpha-d_a^\alpha).
\]
Hence, for $\alpha>0$, $R_\alpha(G') < R_\alpha(G)$ whenever 
$d_y > d_a$ and $d_x > d_b$, or by obvious symmetry, $d_y < d_a$ and $d_x < d_b$. In other words, if the new neighbors of the dissolved pair $x,y$ have both higher or both lower degrees than the original neighbors $y,x$, the Randi\'c index decreases after the switching operation. Conversely, if one vertex of the dissolved pair finds a higher-degree neighbor while the other one finds a lower-degree neighbor, as compared their original neighbors, the Randi\'c index increases. 
An example is depicted in Fig.~\ref{fig:switch}.

\begin{figure}
\includegraphics[scale=0.5]{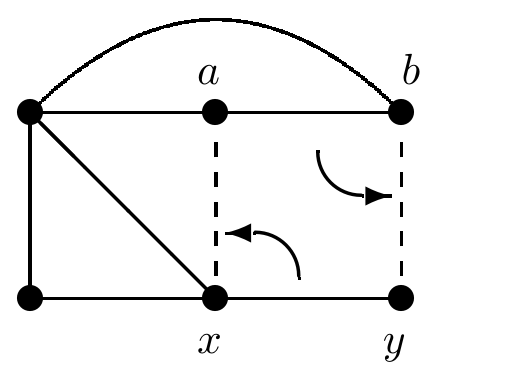}
	\caption{\label{fig:switch} An edge switching operation that increases the Randi\'c index without changing the degree sequence. Here, the edges $xy,ab$ are replaced by $xa,yb$. The reverse operation 
yields an example of an edge switch that decreases the Randi\'c index.}
\end{figure}

So, what does the graph having the largest Randi\'c index for a given degree sequence look like? 
Our next result shows that 
it has a BFD-ordering structure. 

\begin{theorem} \label{theoremrandic} 
	Let $\alpha>0$ and let $G$ be the graph with the largest Randi\'c 
	index $R_\alpha$ in the class $\mathcal{C}_{\pi}$.
	Then there exists a BFD-ordering of $V(G)$ that is consistent with
	its degree sequence $d=(d_{1},\ldots,d_{n})$ in such a way that $d_u>d_v$
	implies $u\prec v$. 
\end{theorem}

\begin{proof}
Let $\alpha>0$ and let $G$ be the graph with the largest Randi\'c 
index $R_\alpha$ in the class $\mathcal{C}_{\pi}$.
Let $\pi=(d_0 \ge \cdots \ge d_{n-1})$ be the  non-increasing degree sequence of $G$. Create another graph $G'$ with the same degree sequence and BFD-ordering by breadth-first search, as described in Section~\ref{sec:bfd}:
Choose the largest degree of $\pi$ as root $v_0$ in layer 0; append all neighbors $v_1,\ldots, v_{d(v_0)}$ of $v_0$ to
the ordered list; these neighbors are ordered such that $u \prec v$ whenever $d_u > d_v$,
or $d_u = d_v$ (in the remaining case the ordering can be arbitrary).
Then continue recursively with all degrees $d_1,d_2,\ldots$ until all degrees of $\pi$ are processed.
This yields the graph $G'$ with BFD-ordering. We show that either $G$ has a BFD-ordering or we get another graph $G'$ by switching the related edges that $G$ has BFD-ordering and $R_\alpha(G)\le R_\alpha(G')$. In each recursive step we compare our graph $G$ and $G'$ if the ordering of is not right we switch with the proper edges and proof show that it is always possible. 
 
We have to show that $u\prec v$ implies $d_u\ge d_v$ for all $u,v \in V(G)$. 
Let $d_v$ be the largest degree in $\pi$ that exists a vertex $x$
with $x \prec v$ but $d_x < d_v$ in $G$. It means that for all degrees $d>d_v$, the vertices $s$ with $d_s>d_v$ are in the right order with respect to the ordering $\prec$.
Let $a$ and $b$ be parents of $x$ and $v$, respectively. If $a=b$, then $x$ and $v$ have the same parent and we can simply change their order and get a right order for $x$ and $v$. Now we can assume that $a$ and $v$ are not adjacent.
Firstly, observe that $d_z\le d_a$ for all neighbors $z$ of $v$; otherwise recursive BFD-ordering process must handle $v$  for a neighbor $z$ with $d_z > d_a$ as a 
\emph{v must have come before x  in the ordering}
parent for $v$  in an earlier time and it must $v\prec x$, a contradiction.\\
Secondly, $v$ has always a neighbor $p$ that is not adjacent to $x$, since $d_x<d_v$. Now 
we are ready to switch (replace) the edges $ax$ and $pv$ with edges $av$ and $xp$ and getting the graph $G'$ and $R_\alpha(G)\le R_\alpha(G')$, since $\Delta R_\alpha$ equals to 
$d_a^\alpha d_v^\alpha+d_x^\alpha d_p^\alpha-d_a^\alpha d_x^\alpha-d_p^\alpha d_v^\alpha =(d_a^\alpha -d_p^\alpha) (d_v^\alpha-d_x^\alpha)\ge 0$, where $p$ is neighbor of $v$ by above observation $d_p\le d_a$.\\
For the case $R_\alpha(G)<R_\alpha(G')$, we get a contradiction to be largest of $R_\alpha(G)$.
Observe that if $R_\alpha(G)=R_\alpha(G')$, then $d_a=d_p$ must hold. By recursive BFD-ordering construction and assumption, it holds $a\prec p$, otherwise process handle $v$ as parent of $p$ before $a$ and get $v\prec x$, a contradiction. Now it follows that $v\prec x$ and $x$ get a new ordering depending on $p$ in $G'$ and $G'$ has more vertices in the right order than $G$. So we can continue the process with $G'$.

It remains to show that after replacing the edges $ax$ and $vp$ with $av$ and $xp$, the graph $G'$ remains connected. 
Since $G$ is connected, there exists a simple
path $P_{av} =(a,\ldots,t, v)$ from $a$ to $v$. 
There are four cases to consider:\\
(1) If $ax \in P_{av}$ and $xt \notin E(G)$, then we set $p = t$.\\
(2) Else, if $ax \notin P_{av}$ and $xt \in E(G)$, then we set $p$ to one of the neighbors of $v$ that is not adjacent to $x$.\\
(3) Else, if $ax \in P_{av}$ and all neighbors not equal $t$ are adjacent to $x$, then $t$ cannot be adjacent to $x$ and we set $p = t$.\\
(4) Else, $ax \in P_{av}$ and there exists a neighbor $s$ of $v$, $s\not= t$, with $xs \notin  E(G)$. Then we set $p = s$.\\
In all cases $G'$ remains connected. 
This completes the proof.
\end{proof}

\begin{remark}
While Theorem~\ref{theoremrandic} gives a necessary condition, finding a sufficient condition for the maximization of Randi\'c index is still an open problem. 
Li et al.~\cite{Li-Doyle05} give a heuristic algorithm for a sufficient condition based on building a partial graph incrementally
in a greedy way. 
However, for some degree sequences this construction can result in a disconnected graph.
There do exist degree sequences,
called \emph{forcibly connected}, such that all graphs having this degree sequence are connected.
A few sufficient conditions are known for a degree sequence to be forcibly connected; 
we mention here a result due to Bondy \cite{Choudum91}: 
A degree sequence $(d_0,\ldots,d_{n-1})$ 
is forcibly connected if $d_{i} \ge n-i$ for every $i\ge d_0+2$. 
For a forcibly connected degree sequence, the heuristic algorithm of 
\cite{Li-Doyle05} computes the graph with maximum
Randi\'c index, and by Theorem~\ref{theoremrandic} this graph has a BFD-ordering.
However, 
for the special case of trees we are able to extend Theorem~\ref{theoremrandic} by providing also a sufficient condition, as presented in the next section.	
\end{remark}

It has been shown that the degree distribution
(or the degree sequence) of a graph does not determine the spectral gap of the graph Laplacian matrix \cite{Atay;Biyikoglu;Jost2006a}. Here we
also notice that in general, by using suitable switching operations,
we can get series of graphs with the same degree sequence and different
(dis)assortativity. Therefore, degree distribution and more generally
degree sequence do not in general suffice to determine degree correlations.

\section{Maximal-entropy and maximal-Randi\'c trees}  \label{sec:trees}

For the special case of trees, the property of BFD-ordering completely 
characterizes the property of having 
maximum entropy or maximum Randi\'c index for given degree sequence, 
thereby establishing a fundamental relationship between the two concepts. 
In the following, the notation 
$\mathcal{T}_\pi$ denotes the set of all trees having degree sequence $\pi$.

\begin{theorem}[\cite{Biyikoglu;Leydold2008}]  \label{thmtreemaxeig}
	A tree $G$ with degree sequence $\pi$ has the largest spectral radius
	in class $\mathcal{T}_{\pi}$ if and only if it is a BFD-tree. $G$
	is then uniquely determined up to isomorphism. The BFD-ordering is
	consistent with the Perron vector $f$ of $G$, in the sense that
	$f(u)>f(v)$ implies $u\prec v$. 
\end{theorem}

One can also compare trees having different degree sequences.
To this end,  
let $\pi=(d_{0},\ldots,d_{n-1})$
and $\pi'=(d'_{0},\ldots,d'_{n-1})$, be two non-identical sequences in
non-increasing order. We write $\pi\lhd\pi'$ if  $\sum_{i=0}^{j}d_{i}\leq\sum_{i=0}^{j}d'_{i}$ for all
$ j=0,\ldots, n-1$. Such an ordering is sometimes called \emph{majorization}. 

\begin{theorem}[\cite{Biyikoglu;Leydold2008}] \label{thmmajor} 
	Let $\pi$ and $\pi'$ be two distinct degree sequences of trees with
	$\pi\lhd\pi'$, and let $G$ and $G'$ be trees having the largest spectral radii in classes $\mathcal{T}_{\pi}$ and $\mathcal{T}_{\pi'}$,
	respectively. Then $\lambda(G)<\lambda(G')$. 
\end{theorem} 

Theorem \ref{thmmajor} can be used to show that,
among all trees with $n$ vertices and $k$ leaves, the largest spectral radius belongs to star-like graphs%
\footnote{The kind of ``star-like'' graph mentioned here is 
%The graphs
obtained from a path
of $n-k+1$ vertices by adding $k-1$ leaves to the first vertex.}
with paths of almost the same lengths attached to each of its $k$ leaves \cite{Biyikoglu;Leydold2008}.

For the Randi\'c index, we have the following counterpart to Theorem~\ref{thmtreemaxeig}. It is a direct consequence of Theorem~\ref{theoremrandic}, and its proof is analogous to that of Theorem~\ref{thmtreemaxeig} in \cite{Biyikoglu;Leydold2008} 
(see also
\cite{Delorme;Favaron;Rautenbach2003}).

\begin{theorem} 
	Let $\alpha>0$. A tree $G$ with degree sequence $\pi$ 
	has the largest Randi\'c index $R_\alpha$ in class 
	$\mathcal{T}_{\pi}$
	if and only if it is a BFD-tree. $G$ is then uniquely determined
	up to an isomorphism.
\end{theorem}

\section{The normalized Randi\'c function} \label{sec:norm-Randic}

In this section we indicate further relations between the Randi\'c index and the concept of entropy. To this end, we also define a normalized version of the Randi\'c index.

So far, when discussing the Randi\'c index $R_\alpha$ 
and its relations to graph entropy and degree assortativity, 
the precise value of $\alpha$ was not specified. 
While most of the observations hold for any $\alpha$,
it is not clear whether a particular value of $\alpha$ has a special meaning.
Even the whole set $R_\alpha$, regarded as a function of $\alpha \in \mathbb{R}$, is not very enlightening---$R_\alpha$ is simply a monotone increasing function of $\alpha$
that grows unbounded as $\alpha\to\infty$.
Therefore, 
we introduce a \emph{normalized Randi\'c function}:
\begin{equation} \label{norm-randic}
	\bar{R}_{\alpha}:=\frac{\sum_{uv\in E}(d_{u}d_{v})^{\alpha}}
	{\sum_{v\in V}(d_{v})^{2\alpha}}.
\end{equation}

The motivation for 
this definition 
is threefold.
First, by Proposition \ref{prop:randic-bound}, 
the spectral radius 
is an upper bound for $\bar{R}_{\alpha}$:
\begin{equation} \label{norm-randic-bound}
	\bar{R}_{\alpha} \le \tfrac{1}{2}\lambda(G),
 	\quad \forall \alpha\in\mathbb{R}.
\end{equation}
Thus, $\bar{R}_{\alpha}$ 
is a  
bounded continuous function of $\alpha \in \mathbb{R}$. 

Second, for a $k$-regular graph, where each vertex has the same degree $k$, the Randi\'c index is $R_\alpha =  |E| k^{2\alpha}$.
It is disappointing that this quantity depends on $\alpha$, namely, on how one chooses to scale the degree correlations, although the degree distribution in the graph is completely homogeneous. 
On the other hand, the normalized Randi\'c for regular graphs has the constant value $\bar{R}_\alpha = |E|/|V| = \frac{1}{2} k$, independently of $\alpha$.
 
Third, since $\bar{R}_{\alpha}$ is bounded, a particular value of $\alpha$ 
now naturally comes to the forefront, namely the value $\alpha=\alpha^*$ for which $\bar{R}_{\alpha}$ assumes its maximum. 
Note that $2\bar{R}_{0}=d_\mathrm{avg}$; 
so, the maximum value $\bar{R}_{\alpha^*}$ is bounded from below by half the average degree.
This and \eqref{norm-randic-bound} give $\bar{R}_{\alpha^*}$  a place within the inequalities \eqref{deg-lambda}:
\begin{equation}	\label{deg-lambda2}
	d_{\min}\le 
	d_\mathrm{avg} \le 2\bar{R}_{\alpha^*} \le \lambda(G)
	\le d_{\max}.
\end{equation}
Moreover, by \eqref{H} and \eqref{norm-randic-bound} we have a relation to graph entropy:
$$
	\log \bar{R}_{\alpha^*} \le \log \lambda(G)-1 = H(G) - 1.
$$

We shall in fact prove a more direct relation to entropy, as follows.
Consider two natural 
probability measures on the edge set $E$ and the vertex set $V$, respectively, defined by (cf. \eqref{eq:pe})
\begin{equation} \label{pe}
	\begin{aligned}
	p_{E}^{\alpha}(uv) &=\frac{(d_{u}d_{v})^{\alpha}}{\sum_{uv\in E}(d_{u}d_{v})^{\alpha}},  \qquad \\
	p_{V}^{\alpha}(v) &=\frac{(d_{v}^{2})^{\alpha}}{\sum_{v\in V}(d_{v}^{2})^{\alpha}}.
	\end{aligned}
\end{equation}
Associated with these measures are the entropy expressions
\begin{equation} \label{HEV}
\begin{aligned}
	H_{E}^{\alpha} & =-\!\!\sum_{uv\in E}p_{E}^{\alpha}(uv)\,\log p_{E}^{\alpha}(uv),  \; \\
	H_{V}^{\alpha} & =-\!\sum_{v\in V}p_{V}^{\alpha}(v)\,\log p_{V}^{\alpha}(v).
\end{aligned}
\end{equation}
With this notation, we can state the following result
as a fundamental relation between the entropies and the normalized Randi\'c function.

\begin{theorem}
The logarithm of the normalized Randi\'c function given by \eqref{norm-randic} satisfies
	\begin{equation}  \label{eq:lograndic}
		\frac{\partial}{\partial \alpha} \log \bar{R}_\alpha = 
		\frac{1}{\alpha} (\log \bar{R}_\alpha + H^\alpha_V - H^\alpha_E).
	\end{equation}
Consequently, at the point $\alpha=\alpha^*$ where $\bar{R}_{\alpha}$ reaches its maximum, 
\begin{equation}  \label{alphastar}
	\log \bar{R}_{\alpha^*} =  H^{\alpha^*}_E - H^{\alpha^*}_V .
\end{equation}	
%}
\end{theorem}

\begin{widetext}
\begin{proof}
Differentiating \eqref{norm-randic} with respect to $\alpha$, we have
%\onecolumngrid

\begin{align*}
%\begin{split}
 \frac{\partial \bar{R}_{\alpha}}{\partial\alpha} 
& =     \left[\sum_{E}\frac{(d_{u}d_{v})^{\alpha} \log(d_{u}d_{v})}{\sum_{E}(d_{u}d_{v})^{\alpha}}  -  \sum_{V}\frac{d_{v}^{2\alpha} \log(d_{v}^{2})}{\sum_{V}d_{v}^{2\alpha}}\right]  \times 
%\\
 \frac{\sum_{E}(d_{u}d_{v})^{\alpha}}{\log e \sum_{V}d_{v}^{2\alpha}}  
%\end{split}
\\
&   = \frac{1}{\alpha \log e} \left[\sum_{E}\frac{(d_{u}d_{v})^{\alpha} \log(d_{u}d_{v})^{\alpha}}{\sum_{E}(d_{u}d_{v})^{\alpha}}-\sum_{V}\frac{d_{v}^{2\alpha} \log(d_{v}^{2\alpha})}{\sum_{V}d_{v}^{2\alpha}}\right]\bar{R}_{\alpha}
\end{align*}
%\end{widetext}
By \eqref{HEV}, the first term inside brackets equals
\[
	\sum_{E}\frac{(d_{u}d_{v})^{\alpha}}{\sum_{E}(d_{u}d_{v})^{\alpha}}\log(d_{u}d_{v})^{\alpha} = -H_{E}^{\alpha}+\log\sum_{E}(d_{u}d_{v})^{\alpha}.
\]
Similarly, the second term 
inside the brackets equals
$-H_{V}^{\alpha}+\log\sum_{V} d_{v}^{2\alpha}$.
Combining, %gives
\begin{align*}
\frac{\partial}{\partial\alpha}\bar{R}_{\alpha} 
 & =  \frac{1}{\alpha \log e}  \times 
 \left(H_{V}^{\alpha} \! - \! H_{E}^{\alpha}+\log\sum_{E}(d_{u}d_{v})^{\alpha} \! -\log\sum_{V}d_{v}^{2\alpha}\right)\! \bar{R}_{\alpha},
\end{align*}
which yields
\begin{equation*}
\frac{\partial}{\partial \alpha} (\log \bar{R}_\alpha)
=
 \log e \cdot \frac{1}{\bar{R}_\alpha} \frac{\partial}{\partial\alpha}\bar{R}_{\alpha}  = \frac{1}{\alpha} \left( H_{V}^{\alpha}-H_{E}^{\alpha}+\log\bar{R}_{\alpha}\right),
\end{equation*}
and proves \eqref{eq:lograndic}.
Since at the maximum value of $\log\bar{R}_{\alpha}$ (and hence of $\bar{R}_{\alpha}$) 
the left-hand side is zero, \eqref{alphastar} follows. 
\end{proof}
\end{widetext}

\begin{figure*}[tbh]
\centering
\includegraphics{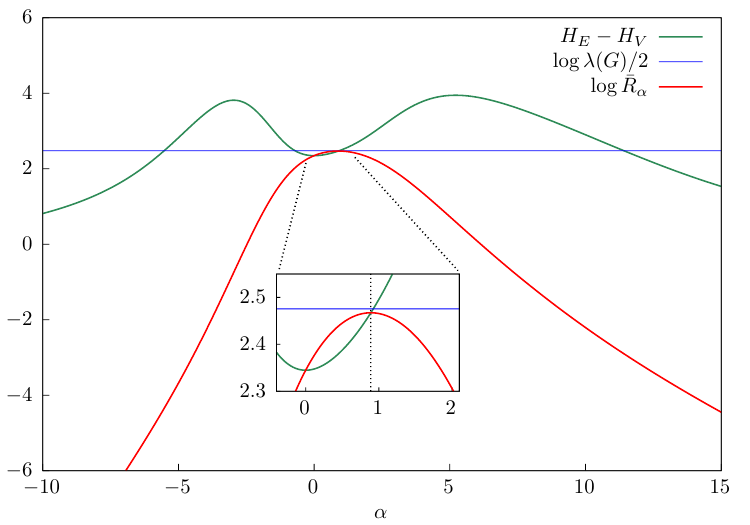}
\caption{Various entropic quantities calculated for a random graph $G$. The upper curve (green) is the difference of Shannon entropies $H_E^\alpha - H_V^\alpha$, and the lower curve (red) shows the logarithm of the normalized Randi\'c function $\bar{R}_\alpha$, as a function of $\alpha$. The horizontal line (blue) marks the value of $\log(\lambda(G)/2) = H(G)-1$, 
where $H(G)$ is the topological entropy of $G$, 
which is an upper bound for $\log \bar{R}_\alpha$. The inset shows a blowup of the region where the curves meet. The maximum value of $\bar{R}_\alpha$ is attained at $\alpha^* \approx 0.89$ (vertical line in the inset).\label{fig:entropy}
}
\end{figure*}

We give a numerical example
illustrating how various entropic concepts come together. 
In Fig.~\ref{fig:entropy} we plot the quantities  $\log \bar{R}_\alpha$ and $H_E^\alpha - H_V^\alpha$ as functions of $\alpha$, calculated for an Erd\H os-R\'enyi random graph on 50 vertices where the probability of an edge between a pair of vertices is $0.2$.  
It can be seen that at its maximum point $\alpha^*\approx 0.89$, $\log \bar{R}_\alpha$ comes very close to its theoretical upper bound $\log\lambda(G)/2$ given in \eqref{deg-lambda2}. 
Moreover, at this point $\log \bar{R}_\alpha$ equals $H_E^\alpha - H_V^\alpha$.

We note that
the curves $\log \bar{R}_\alpha$ and $H_E^\alpha - H_V^\alpha$ also intersect at $\alpha=0$, in addition to $\alpha^*$.
This follows from the fact that the probabilities $p_E^0$ and $p_V^0$ are uniform distributions over the edge and vertex sets; so
$H_E^0=\log|E|$, $H_V^0=\log|V|$, and %therefore 
$$\log \bar{R}_0 = \log(d_\mathrm{avg}/2)= H_E^0-H_V^0.$$ 
Moreover, 
%the curve 
$H_E^\alpha - H_V^\alpha$ has a local extremum at $\alpha=0$,
because the Shannon entropies $H_E^\alpha$ and $H_V^\alpha$ are maximum for uniform probability distributions, i.e.~for $\alpha=0$; 
so 
$ \frac{\partial}{\partial \alpha}(H_E^\alpha -H_V^\alpha) |_{\alpha=0}=0.$
Thus, the curves $\log \bar{R}_\alpha$ and $H_E^\alpha - H_V^\alpha$ have a reciprocal relationship at the two intersection points 
$\alpha=0$ and $\alpha = \alpha^*$:
at the former point $H_E^\alpha - H_V^\alpha$ and at the latter point the $\log \bar{R}_\alpha$ has a local extremum (inset in Fig.~\ref{fig:entropy}).
Consequently, for $\alpha \in (0,\alpha^*)$ the various entropic quantities are related by 
\[
	H_E^\alpha - H_V^\alpha < \log \bar{R}_\alpha < \log \lambda(G).
\]

\section{Numerical Results on real network data}

In this section we present calculations on real network data. 
We have chosen 13 networks that are undirected,
simple, unweighted, and connected, as we deal with such graphs in this paper. 
All networks are taken from \href{https://networkrepository.com}{\emph{The Network Repository}} \cite{nr-aaai15} 
and are listed in Table~\ref{tab:networks} using the same names as in the repository.
The prefixes refer to the various collections, namely, bio (biological networks), ca (collaboration networks),
email (email networks), ia (interaction networks), rt (retweet networks), road (road networks), soc (social networks), socfb (facebook networks), tech (technological networks), web (web graphs).

For each network data and its corresponding graph, we construct a BFD-ordered graph, as explained in Section~\ref{sec:bfd}.
In the construction steps, in case there are several candidates with the same degree for the next vertex, we always pick one that is not yet in BFD; this guarantees that the BFD-ordered graph is connected.
Since the degree sequence is not altered in the BFD construction, the network and its BFD version have identical degree information (columns 2--5 in the table). 
However, the Randi\'c  indices $R$ and the spectral radii $\lambda$ are significantly different.
We list the Randi\'c index $R_\alpha$ and the normalized Randi\'c index $\bar{R}_\alpha$  for values of $\alpha\in\{0.5,1,2\}$ in columns 6--11, and  calculate $\lambda(G)/2$ in the last column.
%The results are summarized in Table \ref{tab:networks}, 

The calculations in Table \ref{tab:networks} illustrate several key points of the paper. 
Firstly, it is seen that the Randi\'c index $R$ can assume very large values in large networks, whereas the normalized Randi\'c $\bar{R}$ introduced in Section \ref{sec:norm-Randic} provides a more convenient magnitude.
Indeed, $R_\alpha$ increases monotonically with $\alpha$ without bound, as mentioned in Section \ref{sec:norm-Randic}, while the normalized index $\bar{R}_\alpha$ is bounded from above. 
The values in the table also demonstrate the upper and lower bounds on $\bar{R}_\alpha$ given in \eqref{norm-randic-bound} and \eqref{deg-lambda2}.
Finally, for each network, the BFD version keeps the degree sequence unchanged but markedly increases the Randi\'c indices as well as the graph entropy $\lambda(G)$, 
in accordance with Theorem \ref{theoremrandic}.

%\begingroup
%\squeezetable
\begin{table*}[tbh]
\setlength{\tabcolsep}{1.5mm}
\begin{tabular}{cccccccccccc}
\toprule 
Network & $|V|$ & $|E|$ & $d_{avg}$ & $d_{\max}$ & $R_{0.5}$ & $\bar{R}_{0.5}$ & $R_{1}$ & $\bar{R}_{1}$ & $R_{2}$  & $\bar{R}_{2}$ & $\lambda(G)/2$\tabularnewline
\midrule 
soc-karate & \multirow{2}{*}{34} & \multirow{2}{*}{78} & \multirow{2}{*}{4.588} & \multirow{2}{*}{17} & 499.50 & 3.2019 & 3640 & 3.0033 & 2.704$\times10^{5}$ & 1.4009 & 3.3628\tabularnewline
BFD &  &  &  &  & 530.55 & 3.4010 & 4440 & 3.6634 & 4.625$\times10^{5}$ & 2.3963 & 3.7878\tabularnewline
\midrule 
soc-dolphins & \multirow{2}{*}{62} & \multirow{2}{*}{159} & \multirow{2}{*}{5.129} & \multirow{2}{*}{12} & 1029.6 & 3.2379 & 7313 & 3.3794 & 4.454$\times10^{5}$ & 3.0363 & 3.5968\tabularnewline
BFD &  &  &  &  & 1061.3 & 3.3373 & 8036 & 3.7135 & 5.726$\times10^{5}$ & 3.9034 & 3.9385\tabularnewline
\midrule 
road-chesapeake & \multirow{2}{*}{39} & \multirow{2}{*}{170} & \multirow{2}{*}{8.718} & \multirow{2}{*}{33} & 1894.1 & 5.5710 & 2.347$\times10^{4}$ & 5.2906 & 4.626$\times10^{6}$ & 2.0854 & 5.7479\tabularnewline
BFD &  &  &  &  & 1954.5 & 5.7486 & 2.627$\times10^{4}$ & 5.9216 & 6.791$\times10^{6}$ & 3.0610 & 6.1670\tabularnewline
\midrule 
socfb-Trinity100 & \multirow{2}{*}{2613} & \multirow{2}{*}{111996} & \multirow{2}{*}{85.722} & \multirow{2}{*}{404} & 1.359$\times10^{7}$ & 60.700 & 1.912$\times10^{9}$ & 65.959 & 5.230$\times10^{13}$ & 55.376 & 67.913\tabularnewline
BFD &  &  &  &  & 1.422$\times10^{7}$ & 63.510 & 2.222$\times10^{9}$ & 76.675 & 8.066$\times10^{13}$ & 85.396 & 89.129\tabularnewline
\midrule 
tech-routers-rf & \multirow{2}{*}{2113} & \multirow{2}{*}{6632} & \multirow{2}{*}{6.277} & \multirow{2}{*}{109} & 1.155$\times10^{5}$ & 8.7107 & 3.156$\times10^{6}$ & 11.005 & 4.991$\times10^{9}$ & 7.0690  & 13.835\tabularnewline
BFD &  &  &  &  & 1.382$\times10^{5}$ & 10.609 & 5.159$\times10^{6}$ & 18.009 & 1.309$\times10^{9}$ & 18.545  & 22.020\tabularnewline
\midrule 
web-polblogs & \multirow{2}{*}{643} & \multirow{2}{*}{2280} & \multirow{2}{*}{7.092} & \multirow{2}{*}{165} & 3.970$\times10^{4}$ & 8.7058 & 9.672$\times10^{5}$ & 8.2941 & 1.136$\times10^{9}$ & 1.2505 & 11.607\tabularnewline
BFD &  &  &  &  & 4.959$\times10^{4}$ & 10.893 & 1.774$\times10^{6}$ & 15.212 & 4.241$\times10^{9}$ & 4.6687  & 17.122\tabularnewline
\midrule 
soc-wiki-Vote & \multirow{2}{*}{889} & \multirow{2}{*}{2914} & \multirow{2}{*}{6.556} & \multirow{2}{*}{102} & 4.169$\times10^{4}$ & 7.1539 & 9.239$\times10^{5}$ & 8.7440 & 1.250$\times10^{9}$ & 4.5062 & 9.9626\tabularnewline
BFD &  &  &  &  & 4.827$\times10^{4}$ & 8.2826 & 1.380$\times10^{6}$ & 13.061 & 2.337$\times10^{9}$  & 8.4245 & 15.231\tabularnewline
\midrule 
rt-twitter-copen & \multirow{2}{*}{761} & \multirow{2}{*}{1029} & \multirow{2}{*}{2.704} & \multirow{2}{*}{37} & 6147.1 & 2.9869 & 6.147$\times10^{4}$ & 3.6352 & 1.725$\times10^{7}$  & 2.3429 & 4.5442\tabularnewline
BFD &  &  &  &  & 7945.6 & 4.4142 & 1.240$\times10^{5}$ & 7.4493 & 5.331$\times10^{7}$ & 7.2424 & 8.8488\tabularnewline
\midrule 
ia-infect-dublin & \multirow{2}{*}{410} & \multirow{2}{*}{2765} & \multirow{2}{*}{13.488} & \multirow{2}{*}{50} & 4.951$\times10^{4}$ & 8.9534 & 1.019$\times10^{6}$ & 9.8464 & 5.698$\times10^{8}$  & 8.8187 & 11.691\tabularnewline
BFD &  &  &  &  & 5.088$\times10^{4}$ & 9.2000 & 1.120$\times10^{6}$ & 10.817 & 7.709$\times10^{8}$  & 11.932 & 12.417\tabularnewline
\midrule 
email-univ & \multirow{2}{*}{1133} & \multirow{2}{*}{5451} & \multirow{2}{*}{9.622} & \multirow{2}{*}{71} & 9.152$\times10^{4}$ & 8.3946 & 1.970$\times10^{6}$ & 9.6681 & 1.508$\times10^{9}$  & 8.0208 & 10.374\tabularnewline
BFD &  &  &  &  & 9.995$\times10^{4}$ & 9.1682 & 2.562$\times10^{6}$ & 12.576 & 2.829$\times10^{9}$ & 15.044 & 15.963\tabularnewline
\midrule 
ca-netscience & \multirow{2}{*}{379} & \multirow{2}{*}{914} & \multirow{2}{*}{4.823} & \multirow{2}{*}{34} & 6344.5 & 3.4707 & 5.563$\times10^{4}$ & 3.7942 & 8.448$\times10^{6}$ & 2.2062 & 5.1877\tabularnewline
BFD &  &  &  &  & 7052.4 & 3.8580 & 8.135$\times10^{4}$ & 5.5484 & 2.210$\times10^{7}$ & 5.7717 & 7.0869\tabularnewline
\midrule 
bio-yeast & \multirow{2}{*}{1458} & \multirow{2}{*}{1948} & \multirow{2}{*}{2.672} & \multirow{2}{*}{56} & 9513.5 & 2.4419 & 6.463$\times10^{4}$ & 2.3276 & 7.430$\times10^{6}$  & 0.4235 & 3.7675\tabularnewline
BFD &  &  &  &  & 1.281$\times10^{4}$ & 3.7044 & 1.875$\times10^{5}$ & 6.8624 & 9.935$\times10^{7}$ & 5.6625 & 9.1127\tabularnewline
\midrule 
bio-celegans & \multirow{2}{*}{453} & \multirow{2}{*}{2025} & \multirow{2}{*}{8.940} & \multirow{2}{*}{237} & 4.500$\times10^{4}$ & 11.110 & 1.666$\times10^{6}$ & 10.259 & 8.204$\times10^{9}$ & 2.1491 & 13.154\tabularnewline
BFD &  &  &  &  & 5.064$\times10^{4}$ & 12.505 & 2.148$\times10^{6}$ & 13.228 & 9.905$\times10^{9}$ & 2.5947 & 16.687\tabularnewline
\bottomrule
\end{tabular}
\caption{Various network quantities calculated on several real networks and their BFD versions: 
Network size $|V|$, number of edges $|E|$, average and maximum degrees, the Randi\'c and the normalized Randi\'c index calculated for values of $\alpha=0.5,1,2$, and half the spectral radius of the graph. The minimum degree is 3 for road-chesapeake and 1 for all other networks.
\label{tab:networks}}
\end{table*}
%\endgroup

\section{Conclusion}

Entropy, as a fundamental notion for quantifying information, complexity, and disorder, is a natural concept for complex networks. In this paper, we have established some fundamental relations between topological graph entropy, which quantifies the dynamical complexity  of processes on the network structure, the Randi\'c index from graph theory, which is related to degree assortativity, and BFD-ordering from theoretical computer science, which is an algorithm for traversing graph data structures. We have introduced a normalized version of the Randi\'c index that exhibits stronger connections to entropic quantities. These rigorous results are hoped to fuel further research in this direction for improving our understanding of complex networks.


\begin{acknowledgments} 

The authors thank Josef Leydold for fruitful discussions.

\end{acknowledgments}

% You should use BibTeX and apsrev.bst for references
% Choosing a journal automatically selects the correct APS
% BibTeX style file (bst file), so only uncomment the line
% below if necessary.
%\bibliographystyle{apsrev4-1}

%\bibliography{graphentropy3}


\end{document}